\documentclass[times,sort&compress,3p]{elsarticle}
\journal{Journal of Multivariate Analysis}
\usepackage[labelfont=bf]{caption}

\usepackage{amsmath,amsfonts,amssymb,amsthm,booktabs,color,epsfig,graphicx,hyperref,url,enumerate}

\newtheorem{definition}{Definition}[section]
\newtheorem{theorem}{Theorem}[section]
\newtheorem{lemma}{Lemma}[section]
\newtheorem{remark}{Remark}

\begin{document}

\begin{frontmatter}

\title{An Identity for Expectations and Characteristic Function of Matrix Variate Skew-normal Distribution with 
Applications to Associated Stochastic Orderings}

\author[a]{Tong Pu}
\author[b]{Narayanaswamy Balakrishnan}
\author[a]{Chuancun Yin\corref{mycorrespondingauthor}}

\address[a]{School of Statistics, Qufu Normal University, Qufu 273165, Shandong, China}
\address[b]{Department of Mathematics and Statistics, McMaster University, Hamilton, Ontario, Canada}
\cortext[mycorrespondingauthor]{Corresponding author. Email address: \url{ccyin@qfnu.edu.cn}\\
Email addresses: \url{putong_hehe@aliyun.com} (T. Pu) and \url{bala@mcmaster.ca} (N. Balakrishnan).
}

\begin{abstract}
We establish an identity for $Ef\left( \boldsymbol{Y}\right) - Ef\left( \boldsymbol{X}\right)$, 
when $\boldsymbol{X}$ and $\boldsymbol{Y}$ both have matrix variate skew-normal distributions and the function $f$ fulfills some weak conditions. 
The characteristic function of matrix variate skew normal distribution is then derived. 
Finally, we make use of it to derive some necessary and sufficient conditions for the comparison of matrix variate skew-normal distributions under six different orders, such as usual stochastic order, 
convex order, increasing convex order, upper orthant order, directionally
convex order and supermodular order. 
\end{abstract}

\end{frontmatter}

\section{Introduction}
Stochastic orders, which are partial orders on a set of random variables, 
provide methods to describe intuitively random variables being larger, riskier or more dependent. 
Stochastic orders of variables have found key applications in such diverse fields as actuarial science, economics, comparison of experiments, reliability analysis and queueing theory. 
Different kinds of stochastic orders possess different properties, characterizations and applications, and 
interested readers may refer to \cite{shaked2007stochastic} and \cite{denuit2006actuarial} for elaborate details. 
\par The goal of this paper is to study integral stochastic orders for matrix variate skew-normal distributions. 
Given two random variables $\boldsymbol{X}$ and $\boldsymbol{Y}$, and a class of measurable functions $\boldsymbol{F}$, 
integral stochastic orders, which were introduced by \cite{muller1997stochastic}, seek the order between $\boldsymbol{X}$ and $\boldsymbol{Y}$ 
by comparing $Ef\left( \boldsymbol{Y}\right)$ and $Ef\left( \boldsymbol{X}\right)$, where $f\in \boldsymbol{F}$. 
Integral stochastic orders include most common stochastic orders 
like usual stochastic order and convex order. 
\cite{muller2001stochastic} provided a general treatment on integral stochastic orders, with the main tool being an indentity 
for $Ef\left( \boldsymbol{Y} \right) - Ef\left( \boldsymbol{X} \right)$, 
when $\boldsymbol{X}$ and $\boldsymbol{Y}$ are multivariate normal random variables. 
This indentity was derived by using Fourier Inversion Theorem and then employing integration by parts.  
Recently, \cite{yin2019stochastic} extended this to the case of multivariate elliptical distributions, 
while \cite{yin2019stochastic} and \cite{Ansari2020Ordering} derived some integral stochastic orderings for multivariate elliptical distribution by using different methods. 
\par The univariate skew-normal distribution was developed by \cite{azzalini1985class}, 
which was subsequently extended to the multivariate case by \cite{azzalini1996multivariate}. 
This distribution presents a mathematically tractable
extension of the multivariate normal distribution, 
and accommodates skewness in the model. 
The matrix variate extension of the skew-normal distribution was presented in two ways, one by \cite{chen2005matrix}, 
and the other as a natural extension of the multivariate skew-normal distribution of \cite{azzalini1996multivariate}. 
The matrix variate skew-normal distribution used in this paper is based on the latter. 
These two different definitions have been compared by \cite{harrar2008matrix}. 
All these works primarily defined the distribution and derived the means, covariances, moments, moment generating functions 
and some other properties. 
The characteristic functions of univariate and multivariate skew-normal distributions were derived by \cite{kim2011characteristic}. 
Stochastic orderings of the univariate skew-normal distribution and the general skew-symmetric
family of distributions were discussed by \cite{azzalini2012some}. 
The characterizations of likelihood ratio order and usual stochastic order of the univariate skew-symmetric 
distribution were discussed by \cite{hurlimann2013likelihood}. 
Recently, \cite{jamali2020comparison} derived some integral stochastic orderings for multivariate skew-normal distribution, while 
\cite{madadi2020family} studied properties and inferential methods for mean-mixtures of multivariate normal distributions.
\par Our work here follows those of \cite{muller1997stochastic} and \cite{jamali2020comparison}. 
Specifically, we extend the results of \cite{jamali2020comparison} for the multivariate skew-normal distributions to the matrix variate case. 
Necessary and sufficient conditions are derived for some important integral stochastic orders for matrix variate skew-normal distributions. 
The main tool used is an identity for $Ef\left( \boldsymbol{Y} \right)-Ef\left( \boldsymbol{X} \right)$, 
when $\boldsymbol{X}$ and $\boldsymbol{Y}$ are both skew-normally distributed matrices. 
We need the characteristic function of skew-normally distributed matrices for establishing this identity. 
For this purpose, we also extend the result of \cite{kim2011characteristic} on the characteristic function of multivariate skew-normal to the matrix variate case. 
\par The rest of this paper is organized as follows. 
In Section \ref{Preliminaries2}, we review skew-normal distributions in multivariate and matrix variate cases and state some of their key properties. 
We also present a brief review of integral stochastic ordering. 
In Section \ref{chara3}, we derive the characteristic function of matrix variate skew-normal distribution. 
In Section \ref{main-results}, we establish an indentity 
for $Ef\left( \boldsymbol{Y} \right) - Ef\left( \boldsymbol{X} \right)$, 
when $\boldsymbol{X}$ and $\boldsymbol{Y}$ are both matrix variate skew-normally distributed random variables. 
In Section \ref{st5}, the ordering results for matrix variate skew-normal distribution, based on the results of Section \ref{main-results}, 
are presented. Finally, Section \ref{CR6} presents some concluding remarks. 
\section{Preliminaries} \label{Preliminaries2}
The following notations will be used throughout this paper. 
We will use lowercase letters, bold lowercase letters and bold capital letters to denote numbers, vectors and matrices, respectively; 
$\Phi \left( \cdot \right)$ and $\phi \left( \cdot \right)$ to denote the cumulative
distribution function and probability density function of the univariate standard normal distribution, respectively; 
and $\Phi_n \left( \cdot ;\boldsymbol{\mu}, \boldsymbol{\Sigma}\right)$ and $\phi_n \left( \cdot ;\boldsymbol{\mu}, \boldsymbol{\Sigma} \right)$ to denote the cumulative
distribution function and probability density function of the multivariate $n$-dimensional normal distribution with mean vector $\boldsymbol{\mu}$ and covariance matrix $\boldsymbol{\Sigma}$, $N_n\left(\boldsymbol{\mu},\boldsymbol{\Sigma}\right)$.
\par 
For any $\mathbf{B} \in \mathbb{R}^{n \times p}$, 
$\mathbf{B}^T$ and $\mathbf{B}^{-1}$ denote the transpose and inverse of $\mathbf{B}$, respectively. 
For $\mathbf{B}=\left(\mathbf{b}_1, \mathbf{b}_2, \dots, \mathbf{b}_p\right)$, 
we will use ${\rm vec}\left( \mathbf{B} \right)=\left( \mathbf{b}_1^T, \mathbf{b}_2^T, \dots, \mathbf{b}_p^T\right)^T$ to denote the matrix vectorization, 
and ${\rm tr}\left( \mathbf{B} \right)$ and ${\rm etr}\left( \mathbf{B} \right)$ to denote the trace and exponential trace of matrix $\mathbf{B}$, respectively. 
For $\mathbf{A} \in \mathbb{R}^{l \times m}$, $\mathbf{B} \in \mathbb{R}^{m \times n}$, $\mathbf{C} \in \mathbb{R}^{n \times p}$ and $\mathbf{D} \in \mathbb{R}^{p \times q}$, we will use $\mathbf{B} \otimes \mathbf{C}$
to denote the Kronecker product of $\mathbf{B}$ and $\mathbf{C}$. It is known that 
${\rm vec}\left( \mathbf{BCD} \right)=\left( \mathbf{D}^T \otimes \mathbf{B} \right) {\rm vec}\left( \mathbf{C}\right) $, 
$\left( \mathbf{A} \otimes \mathbf{B}\right)^T = \mathbf{A}^T \otimes \mathbf{B}^T$, 
 and $\left( \mathbf{A} \otimes \mathbf{C}\right) \left( \mathbf{B} \otimes \mathbf{D}\right) =\left( \mathbf{AB} \otimes \mathbf{CD} \right)$;  
see \cite{magnus2019matrix}. 
\subsection{Skew-Normal Distributions}
\par We first recall the definition of multivariate skew-normal distribution given by \cite{azzalini1999statistical}. 
 Consider a full rank $n \times n$ covariance matrix $\mathbf{\Omega}=\left( \omega_{ij}\right)$, 
and let $\overline{\mathbf{\Omega}}=\boldsymbol{\omega}^{-1} \mathbf{\Omega} \boldsymbol{\omega}^{-1}$ be the associated correlation matrix, where
$\boldsymbol{\omega}=diag\lbrace \sqrt{\omega_{11}}, \dots, \sqrt{\omega_{nn}}\rbrace$. 
\begin{definition} \label{n-dim-sn-def}
(\cite{azzalini1999statistical}) An $n$-dimensional random vector $\boldsymbol{Z}$ is said to have multivariate skew-normal distribution with 
location parameter $\boldsymbol{\mu}$, scale parameter $\mathbf{\Omega}$ and skewness parameter $\boldsymbol{\alpha}$, denoted by $SN_n\left( \boldsymbol{\mu}, \mathbf{\Omega}, \boldsymbol{\alpha}\right)$, 
 if its probability density function is 
 \begin{equation} \label{dens-func-n-dim}
 f_Z(\boldsymbol{z})=2\phi_n\left( \boldsymbol{z};\boldsymbol{\mu}, \mathbf{\Omega}\right) \Phi\left( \boldsymbol{\alpha}^T \boldsymbol{\omega}^{-1} \left( \boldsymbol{z}-\boldsymbol{\mu}\right)\right), \quad \boldsymbol{z} \in \mathbb{R}^n. 
 \end{equation}
\end{definition}
\par The characteristic function of an $n$-dimensional random vector $\boldsymbol{X}$ is given by 
$\Psi_X\left( \boldsymbol{t} \right) =E\left[ {\rm exp}\left( i\boldsymbol{t}^T\boldsymbol{X}\right)\right] $, 
where $\boldsymbol{t}$ is an $n$-dimensional vector and $i=\sqrt{-1}$ is the imaginary number,  
while the characteristic function of an $n \times p$-dimensional random matrix $\boldsymbol{Y}$ is defined as 
$\Psi_Y\left( \boldsymbol{t} \right) =E\left[ {\rm etr}\left( i\mathbf{T}^T\boldsymbol{Y}\right)\right] $, where $\mathbf{T}$ is an $n \times p$-dimensional matrix. 
\cite{kim2011characteristic} derived the characteristic function of multivariate skew-normal distribution as presented below.
\begin{lemma} \label{chara-n-dim}
(\cite{kim2011characteristic}) For an $n$-dimensional random vector $\boldsymbol{Z} \sim SN_n\left( \boldsymbol{\mu}, \mathbf{\Omega}, \boldsymbol{\alpha}\right)$, the characteristic function of $\boldsymbol{Z}$ is 
\begin{equation} \label{chara-func-n-dim}
\Psi_Z(\boldsymbol{t})=2{\rm exp}\left( i\boldsymbol{\mu}^T \boldsymbol{t}-\frac{1}{2} \boldsymbol{t}^T \mathbf{\Omega} \boldsymbol{t}\right) \Phi\left( i \boldsymbol{\delta}^T \boldsymbol{t}\right) ={\rm exp}\left( i\boldsymbol{\mu}^T \boldsymbol{t}-\frac{1}{2} \boldsymbol{t}^T \mathbf{\Omega} \boldsymbol{t}\right) \left( 1+i\tau\left( \boldsymbol{\delta}^T \boldsymbol{t}\right) \right),
\end{equation}
where
\begin{equation} \label{delta-alpha}
\boldsymbol{\delta}=\left( 1+\boldsymbol{\alpha}^T \overline{\mathbf{\Omega}} \boldsymbol{\alpha} \right)^{-1/2} \boldsymbol{\omega} \overline{\mathbf{\Omega}} \boldsymbol{\alpha}, 
\end{equation} 
\begin{equation}
\tau(u)=\sqrt{\frac{2}{\pi}} \int_0^u {\rm exp}\left( z^2/2\right) dz. 
\end{equation}
\end{lemma}
\begin{remark}
The parameter $\boldsymbol{\delta}$ is determined from the 
skewness parameter $\boldsymbol{\alpha}$ and the scale parameter $\mathbf{\Omega}$. 
The relation between $\boldsymbol{\delta}$ and $\boldsymbol{\alpha}$ can be shown as in (\ref{delta-alpha}) and also
\begin{equation}
\boldsymbol{\alpha} =\left( 1-\boldsymbol{\delta}^T \mathbf{\Omega}^{-1} \boldsymbol{\delta}\right)^{-1/2} \boldsymbol{\omega} \mathbf{\Omega}^{-1} \boldsymbol{\delta} . 
\end{equation}
Thus, the parameter $\boldsymbol{\delta}$ can be treated for a different description of skewness, 
and is sometimes used to denote $SN_n\left( \boldsymbol{\mu}, \mathbf{\Omega}, \boldsymbol{\alpha}\right)$ by $SN_n\left(\boldsymbol{\mu}, \mathbf{\Omega}, \boldsymbol{\alpha},\boldsymbol{\delta}\right)$. 
So, we write $SN_n\left(\boldsymbol{\mu}, \mathbf{\Omega}, \boldsymbol{\alpha},*\right)$ or $SN_n\left(\boldsymbol{\mu}, \mathbf{\Omega}, *,\boldsymbol{\delta}\right)$ if 
$\boldsymbol{\delta}$ or $\boldsymbol{\alpha}$ is not important in the discussion to follow.
For a discussion on the parameter $\boldsymbol{\delta}$, one may refer to \cite{azzalini2013skew}. 
\end{remark}

\par  \cite{azzalini1999statistical} proved that multivariate skew-normal distributions are closed under linear transformations. 
In fact, \cite{shushi2018generalized} recently proved that all generalized skew-elliptical distributions are closed under affine transformations. 
\begin{lemma} \label{closed}
(\cite{azzalini1999statistical}) Suppose $\boldsymbol{Y} \sim SN_n\left( \boldsymbol{\mu}, \mathbf{\Omega}, \boldsymbol{\alpha},\boldsymbol{\delta}\right)$. 
Let $\boldsymbol{X}$ be a linear transformation of $\boldsymbol{Y}$, i.e., $\boldsymbol{X} = \mathbf{A}^T \boldsymbol{Y}$, where $\mathbf{A}$ is an $n \times p$ 
full rank matrix. Then, $\boldsymbol{X} \sim SN_p\left( \mathbf{A}^T\boldsymbol{\mu}, \mathbf{\Omega}_X, \boldsymbol{\alpha}_X,\boldsymbol{\delta}_X\right)$, where 
\begin{equation}
\mathbf{\Omega}_X = \mathbf{A}^T \mathbf{\Omega} \mathbf{A},
\end{equation}
\begin{equation} \label{alpha-cal}
\boldsymbol{\alpha}_X = \frac{\boldsymbol{\omega}_X \mathbf{\Omega}_X^{-1} \mathbf{B}^T \boldsymbol{\alpha}}{\sqrt{1+\boldsymbol{\alpha}^T\left( \mathbf{\Omega} - \mathbf{B} \mathbf{\Omega}_X^{-1}\mathbf{B}^T\right)\boldsymbol{\alpha}}},
\end{equation}
\begin{equation} \label{delta-cal}
\boldsymbol{\delta}_X = \mathbf{A}^T \boldsymbol{\delta}
\end{equation}
and
\begin{equation}
\mathbf{B} = \boldsymbol{\omega}^{-1}\boldsymbol{\Omega} \mathbf{A}.
\end{equation}
\end{lemma}
\begin{remark}
Note that (\ref{delta-cal}) can be derived by substituting (\ref{alpha-cal}) into (\ref{delta-alpha}).
\end{remark}
\begin{remark} \label{A-eq-w}
In (\ref{alpha-cal}), if $\mathbf{A} = \boldsymbol{\omega}^{-1}$ or $\mathbf{A} = \boldsymbol{\omega}$, we have $ \boldsymbol{\alpha}_X=\boldsymbol{\alpha}$.
\end{remark}
\par The first-order and second-order moments of multivariate skew-normal distribution have been given by \cite{genton2001moments} as follows. 
\begin{lemma} \label{mean-cov}
(\cite{genton2001moments}) The mean vector and the second-order moment matrix of 
$\boldsymbol{Z} \sim SN_n\left(\boldsymbol{\mu}, \mathbf{\Omega}, *,\boldsymbol{\delta}\right)$ are as follows: 
\begin{equation}
E\left(\boldsymbol{Z}\right)=\boldsymbol{\mu} + \sqrt{\frac{2}{\pi}} \boldsymbol{\delta}, 
\end{equation}
\begin{equation}
E\left(\boldsymbol{Z}\boldsymbol{Z}^T\right)= \mathbf{\Omega} + \boldsymbol{\mu}\boldsymbol{\mu}^T + \sqrt{\frac{2}{\pi}}\left( \boldsymbol{\mu}\boldsymbol{\delta}^T +\boldsymbol{\delta}\boldsymbol{\mu}^T \right). 
\end{equation}
In particular, if $\boldsymbol{\mu}=\boldsymbol{0}$, then we deduce
\begin{equation}
E\left(\boldsymbol{Z}\right)=\sqrt{\frac{2}{\pi}} \boldsymbol{\delta}, 
\end{equation}
\begin{equation}
E\left(\boldsymbol{Z}\boldsymbol{Z}^T\right)= \mathbf{\Omega}. 
\end{equation}
\end{lemma}
The following lemma, called the Fourier Inversion Theorem, can be used to show the relationship between
(\ref{dens-func-n-dim}) and (\ref{chara-func-n-dim}).  
\begin{lemma} \label{Fourier}
Suppose $\boldsymbol{Z}$ is an $n$-dimensional random vector, with density function $f_{Z}\left(z\right)$ and characteristic function $\Psi_{Z}(t)$. 
Then, 
\begin{equation}
f_{Z}\left(\boldsymbol{z}\right) = \left( 2\pi \right)^{-n/2} \int_{\mathbb{R}^n} {\rm exp}\left(-i\boldsymbol{t}^T \boldsymbol{z} \right) \Psi_{Z}(\boldsymbol{t}) d\boldsymbol{t}. 
\end{equation}
If $\boldsymbol{N}$ is an $n \times p$ random matrix, with density function $f_{N}\left(\mathbf{N}\right)$ and characteristic function $\Psi_{N}(T)$, 
then 
\begin{equation}
f_{N}\left(\mathbf{N}\right) = \left( 2\pi \right)^{-np/2} \int_{\mathbb{R}^{n \times p}} {\rm etr}\left( -i\mathbf{T}^T \mathbf{N} \right)\Psi_{N}(\mathbf{T}) d\mathbf{T}. 
\end{equation}
\end{lemma}
The definition of multivariate skew-normal distribution in Definition \ref{n-dim-sn-def} can be extended to the 
matrix variate case. \cite{harrar2008matrix} refer to this distribution as matrix variate skew-normal distribution of 
Azzalini and Dalla Valle type. 
The following definition is basically from \cite{harrar2008matrix}, but we have added a location parameter $\mathbf{M}$ to it.
\begin{definition} \label{def-mat-skew-normal}
(\cite{harrar2008matrix}) An $n \times p$ random matrix $\boldsymbol{Y}$ is said to have matrix variate skew-normal distribution with 
location matrix $\mathbf{M}$, scale matrix $\mathbf{V} \otimes \boldsymbol{\Sigma}$ and 
skewness matrix $\mathbf{B}$, denoted by $\boldsymbol{Y} \sim SN_{n \times p} \left( \mathbf{M}, \mathbf{V} \otimes \boldsymbol{\Sigma}, \mathbf{B}\right)$, 
if its probability density function is given by
\begin{equation}
f_Y\left( \mathbf{Y}\right) = 2 \phi_{n \times p}\left( \mathbf{Y};\mathbf{M},\mathbf{V} \otimes \boldsymbol{\Sigma}\right)\Phi\left({\rm vec}\left(\mathbf{B}\right)^T \boldsymbol{\omega}^{-1} {\rm vec}\left(\mathbf{Y}-\mathbf{M}\right) \right),
\end{equation}
where 
$ \mathbf{M} \in \mathbb{R}_{n \times p}$, $ \mathbf{B} \in \mathbb{R}_{n \times p}$, $\mathbf{V} \in \mathbb{R}_{p \times p}$, 
$ \boldsymbol{\Sigma} \in \mathbb{R}_{n \times n}$, $\boldsymbol{\omega} = \mathbf{v} \otimes \boldsymbol{\sigma}$, and $\mathbf{v}$ and $\boldsymbol{\sigma}$ are defined by
$\overline{\mathbf{V}} = \mathbf{v}^{-1} \mathbf{V} \mathbf{v}^{-1}$ and 
$\overline{\boldsymbol{\Sigma}} = \boldsymbol{\sigma}^{-1} \boldsymbol{\Sigma} \boldsymbol{\sigma}^{-1}$, with  
$\phi_{n \times p}$ denoting the probability density function of an $n \times p$-dimensional matrix variate normal distribution.
\end{definition} 
The following lemma provides a necessary and sufficient condition for matrix variate skew-normal distribution. 
In some works (see \cite{ye2014distribution}), in fact, matrix variate skew-normal distributions are defined as in the following lemma.
\begin{lemma} \label{lemma-mat-skew-normal}
(\cite{harrar2008matrix}) $\boldsymbol{Y} \sim SN_{n \times p} \left( \mathbf{M}, \mathbf{V} \otimes \boldsymbol{\Sigma}, \mathbf{B}\right)$ 
if and only if 
$\boldsymbol{y} = {\rm vec}\left( \boldsymbol{Y}\right) \sim SN_{np} \left( {\rm vec}\left( \mathbf{M} \right), \mathbf{V} \otimes \boldsymbol{\Sigma},{\rm vec}\left( \mathbf{B}\right),*\right)$.
\end{lemma}
\par The folllowing lemma provides an identity for correlation matrix, 
and shows specifically that the Kronecker product of two correlation matrices is a correlation matrix.
\begin{lemma} \label{lemma-1}
We assume matrices $\mathbf{V}_{p \times p}=\left( v_{ij}\right)$ and $\boldsymbol{\Sigma}_{n \times n }=\left( \sigma_{ij} \right)$ are both positive definite and their diagonal elements are positive, 
and let matrix 
$\mathbf{\Omega}_{np \times np}=\mathbf{V} \otimes \boldsymbol{\Sigma}=\left( h_{ij}\right)$. Further, let $\mathbf{v}=diag\lbrace \sqrt{v_{11}} \dots \sqrt{v_{pp}}\rbrace$, 
$\boldsymbol{\sigma}=diag\lbrace \sqrt{\sigma_{11}} \dots \sqrt{\sigma_{nn}}\rbrace$ and $\boldsymbol{\omega}=diag\lbrace \sqrt{\omega_{11}} \dots \sqrt{\omega_{np, np}} \rbrace$. 
Then, $\mathbf{\Omega}\boldsymbol{\omega}^{-1}= \left( \mathbf{V}\mathbf{v}^{-1}\right) \otimes \left( \boldsymbol{\Sigma} \boldsymbol{\sigma}^{-1} \right)$ and 
$\boldsymbol{\omega}^{-1}\mathbf{\Omega}\boldsymbol{\omega}^{-1}= \left( \mathbf{v}^{-1}\mathbf{V}\mathbf{v}^{-1}\right) \otimes \left( \boldsymbol{\sigma}^{-1} \boldsymbol{\Sigma} \boldsymbol{\sigma}^{-1} \right)$. 
\end{lemma}
\begin{proof}
We have 
\begin{equation} \nonumber
\mathbf{\Omega}\boldsymbol{\omega}^{-1}=\left( \mathbf{V} \otimes \boldsymbol{\Sigma}\right) \left( \mathbf{v}^{-1} \otimes \boldsymbol{\sigma}^{-1}\right)=\left( \mathbf{V}\mathbf{v}^{-1}\right) \otimes \left( \boldsymbol{\Sigma} \boldsymbol{\sigma}^{-1} \right),
\end{equation}
\begin{equation} \nonumber
\boldsymbol{\omega}^{-1}\mathbf{\Omega}\boldsymbol{\omega}^{-1}=\left( \mathbf{v}^{-1} \otimes \boldsymbol{\sigma}^{-1}\right) \left( \mathbf{V} \otimes \boldsymbol{\Sigma}\right) \left( \mathbf{v}^{-1} \otimes \boldsymbol{\sigma}^{-1}\right)= \left( \mathbf{v}^{-1}\mathbf{V}\mathbf{v}^{-1}\right) \otimes \left( \boldsymbol{\sigma}^{-1} \boldsymbol{\Sigma} \boldsymbol{\sigma}^{-1} \right). 
\end{equation}
\end{proof}
\begin{remark}
Using Lemma \ref{lemma-1}, it can be shown that $\overline{\mathbf{\Omega}} = \overline{\mathbf{V}} \otimes \overline{\boldsymbol{\Sigma}}$.    
\end{remark}
\subsection{Integral Stochastic Orders}
Integral stochastic orders seek orderings between $\boldsymbol{X}$ and $\boldsymbol{Y}$ 
by comparing $Ef\left( \boldsymbol{Y}\right)$ and $Ef\left( \boldsymbol{X}\right)$. 
\begin{definition} \label{def-F-order}
(\cite{denuit2006actuarial}) Let $\boldsymbol{F}$ be a class of measurable functions $f: \mathbb{R}^{n \times p} \to \mathbb{R}$, and $\boldsymbol{X}$ 
and $\boldsymbol{Y}$ be $n$-dimensional random matrices. 
Then, we say that $\boldsymbol{X} \le_F \boldsymbol{Y}$ if $Ef\left( \boldsymbol{X} \right) \le Ef\left( \boldsymbol{Y} \right)$
holds for all $f \in \boldsymbol{F}$, whenever the expectations are well defined. 
\end{definition}
\cite{jamali2020comparison} and \cite{yin2019stochastic} have discussed $\boldsymbol{F}$-class integral stochastic orders through 
multivariate functions. The definition of $\boldsymbol{F}$-class integral stochastic order is extended here to matrix variate functions as we are interested in comparing random matrices. 
It is easy to verify this extension by means of matrix vectorization. 
\begin{definition}
For a function $f: \mathbb{R}^{n \times p} \to \mathbb{R}$, the difference operator is defined as
\begin{equation}
\Delta_{i, j}^\epsilon f \left( \mathbf{X} \right)=f\left( \mathbf{X} + \epsilon \mathbf{E}_{i, j}\right) - f\left( \mathbf{X}\right), 
\end{equation}
where $\mathbf{E}_{i, j}$ is the $(i, j)$-th unit basis matrix of $\mathbb{R}^{n \times p}$, 
for $i=1, 2, \dots, n$, $j=1, 2, \dots, p$, and $\epsilon > 0$. Then, 
\begin{enumerate}
\item $f$ is supermodular if $\Delta_{k, i}^\epsilon \Delta_{l, j}^\delta f \left( \mathbf{X} \right) \ge 0$ holds for all $\mathbf{X} \in \mathbb{R}^{n \times p}$, $1 \le k < l \le n$, $1 \le i < j \le p$, $\epsilon, \delta > 0$;
\item $f$ is directionally convex if $\Delta_{k, i}^\epsilon \Delta_{l, j}^\delta f \left( \mathbf{X} \right) \ge 0$ holds for all $\mathbf{X} \in \mathbb{R}^{n \times p}$, $k, l=1, 2, \dots, n$, $i, j=1, 2, \dots , p$, $\epsilon, \delta > 0$;
\item $f$ is $\boldsymbol{\Delta}$-monotone if $\Delta_{l_1, j_1}^{\epsilon_1} \Delta_{l_2, j_2}^{\epsilon_2} \dots \Delta_{l_d, j_d}^{\epsilon_d} f \left( \mathbf{X} \right) \ge 0$ holds for all $\mathbf{X} \in \mathbb{R}^{n \times p}$, for any subset \\
$\lbrace \left( l_1, j_1\right), \left( l_2, j_2\right), \dots , \left( l_d, j_d\right)\rbrace $$\subseteq \lbrace\left( k, i\right) |k=1, 2, \dots , n, i=1, 2, \dots , p\rbrace$ and $\epsilon_m>0$, $m=1, \dots , d$. 
\end{enumerate}
\end{definition}
It needs to be pointed out that the italic letters like $\boldsymbol{X},\boldsymbol{Y}$ denote random matrices, 
while roman letters like $\mathbf{X},\mathbf{Y}$ denote constants and nonrandom matrices.
\begin{definition} \label{def-order}
\begin{enumerate}
\item Usual stochastic order: $\boldsymbol{X} \le_{st} \boldsymbol{Y}$ if $\boldsymbol{F}$ is the class of increasing functions;
\item Convex order: $\boldsymbol{X} \le_{cx} \boldsymbol{Y}$ if $\boldsymbol{F}$ is the class of convex functions;
\item Increasing convex order: $\boldsymbol{X} \le_{icx} \boldsymbol{Y}$ if $\boldsymbol{F}$ is the class of increasing convex functions;
\item Upper orthant order: $\boldsymbol{X} \le_{uo} \boldsymbol{Y}$ if $\boldsymbol{F}$ is the class of $\boldsymbol{\Delta}$-monotone functions;
\item Supermodular order: $\boldsymbol{X} \le_{sm} \boldsymbol{Y}$ if $\boldsymbol{F}$ is the class of supermodular functions;
\item Directionally convex order: $\boldsymbol{X} \le_{dcx} \boldsymbol{Y}$ if $\boldsymbol{F}$ is the class of directionally convex functions. 
\end{enumerate}
\end{definition}
\par We consider the functions in the class $\boldsymbol{F}$ which are twice differentiable. 
The gradient vector and the Hessian matrix of a twice differentiable function $f$ are defined by (see \cite{magnus2019matrix})
\begin{equation}
\begin{split}
\nabla f\left( \mathbf{X} \right)=\frac{\partial f \left( \mathbf{X} \right)}{\partial \left( vec \mathbf{X} \right)^T}=\Bigg( & \frac{\partial }{\partial x_{11}}f\left( \mathbf{X} \right), \frac{\partial }{\partial x_{21}}f\left( \mathbf{X} \right), \dots, \frac{\partial }{\partial x_{n1}}f\left( \mathbf{X} \right), \\
&\frac{\partial }{\partial x_{12}}f\left( \mathbf{X} \right), \frac{\partial }{\partial x_{22}}f\left( \mathbf{X} \right), \dots, \frac{\partial }{\partial x_{n2}}f\left( \mathbf{X} \right), \\
& \quad\quad\quad\quad\quad\quad\quad \vdots \\
& \frac{\partial }{\partial x_{1p}}f\left( \mathbf{X} \right), \frac{\partial }{\partial x_{2p}}f\left( \mathbf{X} \right), \dots, \frac{\partial }{\partial x_{np}}f\left( \mathbf{X} \right) \Bigg )
\end{split}
\end{equation}
and
\begin{equation}
H_f\left( \mathbf{X}\right) = \frac{\partial^2 f\left( \mathbf{X}\right) }{\partial vec\mathbf{X}\partial \left( vec\mathbf{X} \right)^T}.
\end{equation}
If we assume $vec\mathbf{X} = \left( x_1, x_2, \dots, x_{np}\right)$, then
\begin{equation}
H_f\left( \mathbf{X} \right)=\left(
\begin{array}{cccc}
\frac{\partial^2 }{\partial x_1 \partial x_1} f\left( \mathbf{X} \right) & \frac{\partial^2 }{\partial x_1 \partial x_2} f\left( \mathbf{X} \right) & \ldots &\frac{\partial^2 }{\partial x_1 \partial x_{np}} f\left( \mathbf{X} \right)\\
\frac{\partial^2 }{\partial x_2 \partial x_1} f\left( \mathbf{X} \right) & \frac{\partial^2 }{\partial x_2 \partial x_2} f\left( \mathbf{X} \right) & \ldots & \frac{\partial^2 }{\partial x_2 \partial x_{np}} f\left( \mathbf{X} \right)\\
\vdots & \vdots & \ddots & \vdots\\
\frac{\partial^2 }{\partial x_{np} \partial x_1} f\left( \mathbf{X} \right) & \frac{\partial^2 }{\partial x_{np} \partial x_2} f\left( \mathbf{X} \right) & \ldots &\frac{\partial^2 }{\partial x_{np} \partial x_{np}} f\left( \mathbf{X} \right)\\
\end{array} \right)_{np \times np}.
\end{equation} 
\begin{theorem}
Let the function $f: \mathbb{R}^{n \times p} \to \mathbb{R}$ be a twice differentiable function. Then, 
\begin{enumerate}
\item $f$ is increasing if and only if $\frac{\partial }{\partial x_{ij}}f\left( \mathbf{X} \right)$ holds for all $\mathbf{X} \in \mathbb{R}^{n \times p}$ and $i=1, 2, \dots , n$, $j=1, 2, \dots , p$;
\item $f$ is convex if and only if $H_f\left( \mathbf{X} \right)$ is positive semi-definite, for all $\mathbf{X} \in \mathbb{R}^{n \times p}$;
\item $f$ is supermodular if and only if $\frac{\partial^2}{\partial X_{kj} \partial X_{li}} f\left( \mathbf{X} \right) \ge 0$ holds for all $\mathbf{X} \in \mathbb{R}^{n \times p}$, $1 \le k, l \le n$, $1 \le i, j \le p$, and $\left( k, j \right) \neq \left( l, i\right)$;
\item $f$ is directionally convex if and only if $\frac{\partial^2}{\partial X_{kj} \partial X_{li}} f\left( \mathbf{X} \right) \ge 0$ holds for all $\mathbf{X} \in \mathbb{R}^{n \times p}$, $1 \le k, l \le n$, $1 \le i, j \le p$. 
\end{enumerate}
\end{theorem}
\begin{proof}
Matrix function $f\left( \mathbf{X}\right) : \mathbb{R}^{n \times p} \to \mathbb{R} $ can be treated as a multivariate
function $f\left( vec\mathbf{X}\right) : \mathbb{R}^{np} \to \mathbb{R} $. So, the results in the theorem can be 
derived from the results for multivariate functions presented by \cite{arlotto2009hessian} and \cite{denuit2002smooth}.
\end{proof}
\section{Characteristic Function of Matrix Skew-normal Distribution} \label{chara3}
The following theorem presents the characteristic function of matrix variate skew-normal distribution.
\begin{theorem} \label{chara}
Let $\boldsymbol{Z} \sim SN_{n \times p} \left(\mathbf{M} , \mathbf{V} \otimes \boldsymbol{\Sigma}, \mathbf{B}\right)$. Then, the
 characteristic function of $\boldsymbol{Z}$ is 
\begin{equation} \label{chara-func}
\Psi_Z(\mathbf{T})=2{\rm etr}\left( i\mathbf{M}^T \mathbf{T}-\frac{1}{2} \mathbf{T}^T \boldsymbol{\Sigma T V}\right) \Phi\left( i\frac{{\rm tr}\left( \mathbf{V} \mathbf{v}^{-1} \mathbf{B}^T \boldsymbol{\sigma}^{-1} \boldsymbol{\Sigma} \mathbf{T} \right)}{\sqrt{1+\mathbf{B}^T \overline{\boldsymbol{\Sigma}} \mathbf{B} \overline{\mathbf{V}}}}\right) ,
\end{equation}
where $\mathbf{v}=diag\lbrace \sqrt{v_{11}} ,\dots, \sqrt{v_{pp}}\rbrace$, 
$\boldsymbol{\sigma}=diag\lbrace \sqrt{\sigma_{11}} ,\dots, \sqrt{\sigma_{nn}}\rbrace$,  
and $i=\sqrt{-1}$. 
\end{theorem}
\begin{proof}
Let $\boldsymbol{z}=vec(\boldsymbol{Z})$ and $\mathbf{\Omega} = \mathbf{V} \otimes \boldsymbol{\Sigma}$. We can then show that $\boldsymbol{z} \sim SN_{np}\left( {\rm vec}\left( \mathbf{M} \right), \mathbf{\Omega},{\rm vec}\left( \mathbf{B} \right)\right)$ from Lemma \ref{lemma-mat-skew-normal}. 
Then, the characteristic function of $\boldsymbol{z}$ is 
\begin{equation} 
\begin{split}
\Psi_{\boldsymbol{z}}\left( \boldsymbol{t} \right)&=2{\rm exp}\left( i {\rm vec}\left( \mathbf{M} \right)^T \boldsymbol{t}-\frac{1}{2} \boldsymbol{t}^T \left( \mathbf{V} \otimes \boldsymbol{\Sigma}\right) \boldsymbol{t} \right) \Phi\left( i\left( \frac{ \left( \mathbf{V} \otimes \boldsymbol{\Sigma}\right) \left( \mathbf{v}^{-1} \otimes \boldsymbol{\sigma}^{-1}\right) {\rm vec}\left(B\right)}{\sqrt{1+{\rm vec}\left(\mathbf{B}\right)^T \left( \overline{\mathbf{V}} \otimes \overline{\boldsymbol{\Sigma}}\right) {\rm vec}\left(\mathbf{B}\right)}} \right)^T\boldsymbol{t}\right)\\
&=2{\rm exp}\left( i {\rm vec}\left( \mathbf{M} \right)^T \boldsymbol{t}-\frac{1}{2} \boldsymbol{t}^T \left( \mathbf{V} \otimes \boldsymbol{\Sigma}\right) \boldsymbol{t} \right) \Phi\left( i \frac{{\rm vec}\left( \boldsymbol{\Sigma} \boldsymbol{\sigma}^{-1} \mathbf{B} \mathbf{v}^{-1} \mathbf{V}\right)^T}{\sqrt{1+{\rm tr}\left(\mathbf{B}^T \overline{\boldsymbol{\Sigma}} \mathbf{B} \overline{\mathbf{V}}\right)}} \boldsymbol{t}\right).
\end{split}
\end{equation}
With $t={\rm vec}\left( \mathbf{T} \right)$, we then have 
\begin{equation} \label{vec-chara}
\begin{split}
\Psi_{\boldsymbol{z}}\left( {\rm vec}\left( \mathbf{T} \right) \right)&=2{\rm etr}\left( i\mathbf{M}^T T-\frac{1}{2} \mathbf{T}^T \boldsymbol{\Sigma T V}\right) \Phi\left( i \frac{{\rm vec}\left( \boldsymbol{\Sigma} \boldsymbol{\sigma}^{-1} \mathbf{B} \mathbf{v}^{-1} \mathbf{V}\right)^T {\rm vec}\left( \mathbf{T}\right)}{\sqrt{1+{\rm tr}\left(\mathbf{B}^T \overline{\boldsymbol{\Sigma}} \mathbf{B} \overline{\mathbf{V}}\right)}}  \right) \\
&=2{\rm etr}\left( i\mathbf{M}^T T-\frac{1}{2} \mathbf{T}^T \boldsymbol{\Sigma T V}\right) \Phi\left( i\frac{{\rm tr}\left( \mathbf{V} \mathbf{v}^{-1} \mathbf{B}^T \boldsymbol{\sigma}^{-1} \boldsymbol{\Sigma} \mathbf{T} \right)}{\sqrt{1+{\rm tr}\left(\mathbf{B}^T \overline{\boldsymbol{\Sigma}} \mathbf{B} \overline{\mathbf{V}}\right)}}\right).
\end{split}
\end{equation}
Due to the fact that 
\begin{equation}
E\left( {\rm etr}\left( i\boldsymbol{Z}^T \mathbf{T}\right)\right) = E\left( {\rm exp}\left( i{\rm vec}\left(\boldsymbol{Z}\right)^T {\rm vec}\left(\mathbf{T}\right)\right)\right),
\end{equation}
we have
\begin{equation}
\Psi_Z(\mathbf{T})=2{\rm etr}\left( i\mathbf{M}^T T-\frac{1}{2} \mathbf{T}^T \boldsymbol{\Sigma T V}\right) \Phi\left( i\frac{{\rm tr}\left( \mathbf{V} \mathbf{v}^{-1} \mathbf{B}^T \boldsymbol{\sigma}^{-1} \boldsymbol{\Sigma} \mathbf{T} \right)}{\sqrt{1+{\rm tr}\left(\mathbf{B}^T \overline{\boldsymbol{\Sigma}} \mathbf{B} \overline{\mathbf{V}}\right)}}\right) 
\end{equation}
as required.
\end{proof}
\begin{remark}
If we set $p=1$ in Theorem \ref{chara}, the characteristic function of 
$\boldsymbol{X} \sim SN_{n \times 1} \left( \mathbf{M} , 1 \otimes \boldsymbol{\Sigma}, \mathbf{B}\right)$ simplifies to 
\begin{equation}
\Psi_X(T)=2{\rm exp}\left( i\mathbf{M}^T T-\frac{1}{2} \mathbf{T}^T \boldsymbol{\Sigma T V}\right) \Phi\left( i\frac{\mathbf{B}^T \boldsymbol{\sigma}^{-1} \boldsymbol{\Sigma} \mathbf{T}}{\sqrt{1+\mathbf{B}^T \overline{\boldsymbol{\Sigma}} \mathbf{B}}}\right).
\end{equation}
If we regard $\boldsymbol{X}$, $\mathbf{M}$ and $\mathbf{B}$ as $n$-dimensional vectors, and regard $ SN_{n \times 1} \left( \mathbf{M} , 1 \otimes \boldsymbol{\Sigma}, \mathbf{B}\right)$ as a multivariate skew-normal distribution, then its characteristic
function is identical to the result in Lemma \ref{chara-n-dim}. 
\end{remark}
\begin{remark}
If we set the skew matrix $\mathbf{B}=\boldsymbol{0}_{n \times p}$, then the matrix variate skew-normal distribution $SN_{n \times p} \left( \mathbf{M}, \mathbf{V}\otimes \boldsymbol{\Sigma}, \mathbf{B}\right)$
will degenerate to a matrix variate normal distribution $N_{n \times p} \left( \mathbf{M}, \mathbf{V}\otimes \boldsymbol{\Sigma}\right)$, with characteristic function ${\rm etr}\left( i\mathbf{M}^T\mathbf{T}-\frac{1}{2} \mathbf{T}^T \boldsymbol{\Sigma T V}\right)$, 
which is indeed the characteristic function of a matrix variate normal distribution. 
\end{remark}
\begin{remark}
If we set
\begin{equation}
\boldsymbol{\delta}=\frac{\left( \mathbf{V} \otimes \boldsymbol{\Sigma}\right) \left( \mathbf{v}^{-1} \otimes \boldsymbol{\sigma}^{-1}\right) {\rm vec}\left(\mathbf{B}\right)}{\sqrt{1+{\rm tr}\left(\mathbf{B}^T \overline{\boldsymbol{\Sigma}} \mathbf{B} \overline{\mathbf{V}}\right)}}=\frac{{\rm vec}\left( \boldsymbol{\Sigma} \boldsymbol{\sigma}^{-1} \mathbf{B} \mathbf{v}^{-1} \mathbf{V}\right) }{\sqrt{1+{\rm tr}\left(\mathbf{B}^T \overline{\boldsymbol{\Sigma}} \mathbf{B} \overline{\mathbf{V}}\right)}}, 
\end{equation}
(\ref{chara-func}) can be expressed as
\begin{equation} \label{sim-chara}
\Psi_Z\left( T \right)=2{\rm exp}\left( i {\rm vec}\left( \mathbf{M} \right)^T {\rm vec}\left( \mathbf{T} \right)-\frac{1}{2} {\rm vec}\left( \mathbf{T} \right)^T \left( \mathbf{V} \otimes \boldsymbol{\Sigma}\right) {\rm vec}\left( \mathbf{T} \right) \right) \Phi\left( i\boldsymbol{\delta}^T {\rm vec}\left( \mathbf{T} \right)\right). 
\end{equation}
Like in the multivariate case, in the ensuing discussion, we write 
$SN_{n\times p} \left( \mathbf{M} , \mathbf{V} \otimes \boldsymbol{\Sigma}, \mathbf{B},\boldsymbol{\delta} \right)$ 
for the matrix variate skew-normal distribution. 
Sometimes, we write $SN_{n\times p} \left( \mathbf{M} , \mathbf{V} \otimes \boldsymbol{\Sigma}, \mathbf{B},* \right)$ 
or 
$SN_{n\times p} \left( \mathbf{M} , \mathbf{V} \otimes \boldsymbol{\Sigma}, *,\boldsymbol{\delta} \right)$ 
if parameters $ \boldsymbol{\delta}$ or $ \mathbf{B}$ are not important in the discussion.
\end{remark}
\begin{remark} \label{share-delta}
Let $\boldsymbol{Z} \sim SN_{n \times p} \left(\mathbf{M} , \mathbf{V} \otimes \boldsymbol{\Sigma}, \mathbf{B},\boldsymbol{\delta}\right)$. Then, we can deduce from Lemma \ref{lemma-mat-skew-normal} that 
${\rm vec} \left(\boldsymbol{Z}\right) \sim SN_{np} ( {\rm vec}\left(\mathbf{M}\right) , \mathbf{V} \otimes \boldsymbol{\Sigma}, {\rm vec}\left(\mathbf{B}\right),\boldsymbol{\delta} )$ and that they share the same parameter $\boldsymbol{\delta}$.
\end{remark}
\section{Main results} \label{main-results}
The next theorem presents an indentity for $ E f\left( \boldsymbol{Y}\right) -E f\left( \boldsymbol{X}\right)$, 
and it extends an identity of \cite{muller2001stochastic} for the case of multivariate normal distribution. 

\begin{theorem} \label{comp}
Suppose the $n \times p$ skew-normal random matrices $\boldsymbol{X}$ and $\boldsymbol{Y}$ are distributed as
\begin{equation} \label{supposexy}
\boldsymbol{X} \sim SN_{n \times p} \left( \mathbf{M}, \mathbf{\Omega}, \mathbf{B},\boldsymbol{\delta} \right), \boldsymbol{Y} \sim SN_{n \times p} \left( \mathbf{M}', \mathbf{\Omega}', \mathbf{B}',\boldsymbol{\delta}' \right), 
\end{equation}
where $\mathbf{M},\mathbf{M}',\mathbf{B},\mathbf{B}' \in \mathbb{R}^{n \times p}$, $\mathbf{\Omega}=\mathbf{V}_{p \times p} \otimes \boldsymbol{\Sigma}_{n \times n}$, $\mathbf{\Omega}'=\mathbf{V}'_{p \times p} \otimes \boldsymbol{\Sigma}'_{n \times n}$, $\mathbf{V}$, $\mathbf{V}'$, $\boldsymbol{\Sigma}$ and $\boldsymbol{\Sigma}'$ are positive definite, 
and $\boldsymbol{\delta}, \boldsymbol{\delta}' \in \mathbb{R}^{np}$. 
Suppose an $n \times p$-dimensional matrix variate skew-normal random variable $\boldsymbol{Z}_\lambda$ is distributed as
\begin{equation}
\boldsymbol{Z}_\lambda \sim SN_{n \times p} \left( \mathbf{M}_\lambda, \mathbf{\Omega}_\lambda, *,\boldsymbol{\delta}_\lambda \right),
\end{equation}
where 
\begin{equation} \nonumber
\mathbf{M}_\lambda=\lambda\mathbf{M}'+\left( 1-\lambda \right) \mathbf{M}, 
\end{equation}
\begin{equation} \nonumber
\mathbf{\Omega}_\lambda=\lambda\mathbf{\Omega}'+\left( 1-\lambda \right) \mathbf{\Omega}, 
\end{equation}
\begin{equation} \nonumber
\boldsymbol{\delta}_\lambda=\lambda\boldsymbol{\delta}'+\left( 1-\lambda \right) \boldsymbol{\delta}. 
\end{equation}
Let the density functions of random matrices $\boldsymbol{X}$, $\boldsymbol{Y}$ and $\boldsymbol{Z}_\lambda$ 
be $\phi_0$, $\phi_1$ and $\phi_\lambda$, respectively, 
and the corresponding characteristic functions be $\Psi_0$, $\Psi_1$ and $\Psi_\lambda$. 
Let $\overline{\phi_\lambda}$ and $\overline{\Psi_\lambda}$ be the density function and characteristic function of 
matrix variate normal distribution with mean matrix $\mathbf{M}_\lambda$ and covariance 
$\mathbf{\Omega}_\lambda - \boldsymbol{\delta}_\lambda \boldsymbol{\delta}_\lambda^T$. 
Suppose the function $f: \mathbb{R}^{n \times p} \to \mathbb{R}$ is twice differentiable and is such that  
\begin{enumerate}
\item $\lim_{x_{li} \to \pm \infty} f\left( \mathbf{X} \right) \phi_\lambda\left( \mathbf{X} \right) = 0$, $\lim_{x_{li} \to \pm \infty} f\left( \mathbf{X} \right) \overline{\phi_\lambda}\left( \mathbf{X} \right) = 0$, $\forall\mathbf{X}\in\mathbb{R}^{n \times p}$, $0 \le \lambda \le 1$, $l=1, 2, \dots, n$ and $i=1, 2, \dots, p$,
\item $\lim_{x_{li} \to \pm \infty} f\left( \mathbf{X} \right) \frac{\partial}{\partial x_{kj}}\phi_\lambda\left( \mathbf{X} \right) = 0$, $\lim_{x_{li} \to \pm \infty} f\left( \mathbf{X} \right) \frac{\partial}{\partial x_{kj}} \overline{\phi_\lambda}\left( \mathbf{X} \right) = 0$, $\forall\mathbf{X}\in\mathbb{R}^{n \times p}$, $0 \le \lambda \le 1$, $k, l=1, 2, \dots, n$ and $i, j=1, 2, \dots, p$,
\item $\lim_{x_{li} \to \pm \infty} \phi_\lambda\left( \mathbf{X} \right) \frac{\partial}{\partial x_{kj}} f\left( \mathbf{X} \right) = 0$, $\lim_{x_{li} \to \pm \infty} \overline{\phi_\lambda}\left( \mathbf{X} \right) \frac{\partial}{\partial x_{kj}} f\left( \mathbf{X} \right) = 0$, $\forall\mathbf{X}\in\mathbb{R}^{n \times p}$, $0 \le \lambda \le 1$, $k, l=1, 2, \dots, n$ and $i, j=1, 2, \dots, p$. 
\end{enumerate}
Then, 
\begin{equation} \label{result1}
\begin{split}
E\left[ f\left( \boldsymbol{Y}\right) - f\left( \boldsymbol{X}\right)\right]=& \int_0^1 \int_{\mathbb{R}^{n \times p}} \left( \left( \nabla f\left( \mathbf{X}\right) {\rm vec}\left( \mathbf{M}' - \mathbf{M}\right)\right) + \frac{1}{2}{\rm tr}\left( \left( \mathbf{\Omega}'-\mathbf{\Omega}\right) H_f\left( \mathbf{X}\right)\right)\right)\phi_\lambda\left(\mathbf{X}\right)\\
&+\frac{2}{\sqrt{2 \pi}}\left( \nabla f\left( \mathbf{X}\right)  \left( \boldsymbol{\delta}' - \boldsymbol{\delta}\right) \right) \overline{\phi_\lambda}\left( \mathbf{X}\right) d\mathbf{X} d\lambda. 
\end{split}
\end{equation}
\end{theorem}
\begin{proof}
Let $g\left( \lambda \right)=\int_{\mathbb{R}^{n \times p}} f\left( \mathbf{X}\right) \phi_\lambda\left( \mathbf{X}\right) d\mathbf{X}$. 
Then, we have $g\left( 0\right) =Ef\left( \boldsymbol{X}\right)$, $g\left( 1\right) =Ef\left( \boldsymbol{Y}\right)$ and 
\begin{equation} \label{E-int}
Ef\left( \boldsymbol{Y}\right) - Ef\left( \boldsymbol{X}\right) = \int_0^1 \frac{\partial}{\partial \lambda} g\left( \lambda \right)d\lambda=\int_0^1 \int_{\mathbb{R}^{n \times p}} f\left( \mathbf{X}\right) \frac{\partial}{\partial \lambda} \phi_\lambda\left( \mathbf{X}\right) d\mathbf{X} d\lambda. 
\end{equation}
According to the Fourier inversion formula in Lemma \ref{Fourier}, we have
\begin{equation} \label{part-dens}
\frac{\partial}{\partial \lambda} \phi_\lambda\left( \mathbf{X}\right) = \frac{1}{\left( 2\pi\right)^{np}} \int_{\mathbb{R}^{n \times p}} {\rm etr}\left( -i\mathbf{T}^T\mathbf{X}\right) \frac{\partial}{\partial \lambda} \Psi_\lambda\left(\mathbf{T}\right)d\mathbf{T}.
\end{equation}
Using Theorem \ref{chara}, we have
\begin{equation} \label{part-chara}
\begin{split}
\frac{\partial}{\partial \lambda} \Psi_\lambda\left(\mathbf{T}\right)&=\frac{\partial}{\partial \lambda} 2{\rm exp}\left( i{\rm vec}\left( \mathbf{M}_\lambda\right)^T {\rm vec}\left( \mathbf{T} \right)-\frac{1}{2} {\rm vec}\left( \mathbf{T} \right)^T \mathbf{\Omega}_\lambda {\rm vec}\left( \mathbf{T} \right) \right) \Phi\left( i\boldsymbol{\delta}_\lambda^T {\rm vec}\left( \mathbf{T}\right)\right)\\
&=\lbrace i {\rm vec}\left( \mathbf{T}\right)^T {\rm vec}\left( \mathbf{M}'-\mathbf{M} \right) - \frac{1}{2}{\rm vec}\left(\mathbf{T}\right)^T \left( \mathbf{\Omega}' - \mathbf{\Omega}\right) {\rm vec}\left(\mathbf{T}\right)\rbrace \Psi_\lambda\left( \mathbf{T} \right) \\
&+\frac{2}{\sqrt{2 \pi}} \left( i\left( \boldsymbol{\delta}'-\boldsymbol{\delta}\right)^T {\rm vec}\left( \mathbf{T}\right) \right) \overline{\Psi_\lambda}\left( \mathbf{T} \right).
\end{split}
\end{equation}
Upon substituting (\ref{part-chara}) into (\ref{part-dens}), we get 
\begin{equation} \label{part-dens-2}
\begin{split}
\frac{\partial}{\partial \lambda} \phi_\lambda\left( \mathbf{X}\right) =&\frac{1}{2} \sum_{k, l=1}^{n} \sum_{i, j=1}^{p} \left( v_{ij}'\sigma_{kl}'-v_{ij}\sigma_{kl}\right) \frac{\partial^2}{\partial x_{ki} \partial x_{lj}} \phi_\lambda\left( \mathbf{X}\right)\\
&-\sum_{k=1}^{n} \sum_{j=1}^{p} \left( m_{kj}'-m_{kj}\right) \frac{\partial}{\partial x_{kj}} \phi_\lambda\left( \mathbf{X}\right) \\
&-\frac{2}{\sqrt{2\pi}} \sum_{k=1}^{n} \sum_{j=1}^{p} \left( \delta_{k+n\left( j-1\right)}'-\delta_{k+n\left( j-1\right)}\right) \frac{\partial}{\partial x_{kj}} \overline{\phi_\lambda}\left( \mathbf{X}\right). 
\end{split}
\end{equation}
Substituting (\ref{part-dens-2}) back into (\ref{E-int}), we obtain
\begin{equation}
\begin{split}
\int_{\mathbb{R}^{n \times p}} f\left( \mathbf{X}\right) \frac{\partial}{\partial \lambda} \phi_\lambda\left( \mathbf{X}\right) d\mathbf{X}=&\frac{1}{2} \sum_{k, l=1}^{n} \sum_{i, j=1}^{p} \left( v_{ij}'\sigma_{kl}'-v_{ij}\sigma_{kl}\right) \int_{\mathbb{R}^{n \times p}} f\left( \mathbf{X}\right) \frac{\partial^2}{\partial x_{ki} \partial x_{lj}} \phi_\lambda\left( \mathbf{X}\right) d\mathbf{X} \\
&+\sum_{k=1}^{n} \sum_{j=1}^{p} \left( m_{kj}'-m_{kj}\right) \int_{\mathbb{R}^{n \times p}} f\left( \mathbf{X}\right) \frac{\partial}{\partial x_{kj}} \phi_\lambda\left( \mathbf{X}\right) d\mathbf{X} \\
&+\frac{2}{\sqrt{2\pi}} \sum_{k=1}^{n} \sum_{j=1}^{p} \left( \delta_{k+n\left( j-1\right)}'-\delta_{k+n\left( j-1\right)}\right) \int_{\mathbb{R}^{n \times p}} f\left( \mathbf{X}\right) \frac{\partial}{\partial x_{kj}} \overline{\phi_\lambda}\left( \mathbf{X}\right) d\mathbf{X}.
\end{split}
\end{equation}
Now, integrating by parts and then using the properties of $f$, we get
\begin{equation} \label{fina}
\begin{split}
\int_{\mathbb{R}^{n \times p}} f\left( \mathbf{X}\right) \frac{\partial}{\partial \lambda} \phi_\lambda\left( \mathbf{X}\right) d\mathbf{X}=&\frac{1}{2} \sum_{k, l=1}^{n} \sum_{i, j=1}^{p} \left( v_{ij}'\sigma_{kl}'-v_{ij}\sigma_{kl}\right) \int_{\mathbb{R}^{n \times p}} \phi_\lambda\left( \mathbf{X}\right) \frac{\partial^2}{\partial x_{ki} \partial x_{lj}} f\left( \mathbf{X}\right) d\mathbf{X} \\
&+\sum_{k=1}^{n} \sum_{j=1}^{p} \left( m_{kj}'-m_{kj}\right) \int_{\mathbb{R}^{n \times p}} \phi_\lambda\left( \mathbf{X}\right) \frac{\partial}{\partial x_{kj}} f\left( \mathbf{X}\right) d\mathbf{X} \\
&+\frac{2}{\sqrt{2\pi}} \sum_{k=1}^{n} \sum_{j=1}^{p} \left( \delta_{k+n\left( j-1\right)}'-\delta_{k+n\left( j-1\right)}\right) \int_{\mathbb{R}^{n \times p}} \overline{\phi_\lambda}\left( \mathbf{X}\right) \frac{\partial}{\partial x_{kj}} f\left( \mathbf{X}\right) d\mathbf{X}\\
=&\int_{\mathbb{R}^{n \times p}} \left( \left( \nabla f\left( \mathbf{X}\right) {\rm vec}\left( \mathbf{M}' - \mathbf{M}\right)\right) + \frac{1}{2}{\rm tr}\left( \left( \mathbf{\Omega}'-\mathbf{\Omega}\right) H_f\left( \mathbf{X}\right)\right)\right)\phi_\lambda\left(\mathbf{X}\right)\\
&+\frac{2}{\sqrt{2 \pi}}\left( \nabla f\left( \mathbf{X}\right) \left( \boldsymbol{\delta}' - \boldsymbol{\delta}\right) \right) \overline{\phi_\lambda}\left( \mathbf{X}\right) d\mathbf{X}.
\end{split}
\end{equation}
Substitution of (\ref{fina}) into (\ref{E-int}) yields the result in (\ref{result1}). 
\end{proof}
\begin{theorem} \label{coll}
Suppose the $n \times p$ skew-normal random matrices $\boldsymbol{X}$ and $\boldsymbol{Y}$ 
and the function $f$ satisfy the conditions in Theorem \ref{comp}. If
\begin{enumerate}
\item $\sum_{k, l=1}^{n} \sum_{i, j=1}^{p} \left( v_{ij}'\sigma_{kl}'-v_{ij}\sigma_{kl}\right) \frac{\partial^2}{\partial x_{ki} \partial x_{lj}} f\left( \mathbf{X}\right) \ge 0$,
\item $\sum_{k=1}^{n} \sum_{j=1}^{p} \left( m_{kj}'-m_{kj}\right) \frac{\partial}{\partial x_{kj}} f\left( \mathbf{X}\right) \ge 0$,
\item $\sum_{k=1}^{n} \sum_{j=1}^{p} \left( \delta_{k+n\left( j-1\right)}'-\delta_{k+n\left( j-1\right)}\right) \frac{\partial}{\partial x_{kj}} f\left( \mathbf{X}\right) \ge 0$, 
\end{enumerate}
then $Ef\left( \boldsymbol{Y}\right) \ge Ef\left( \boldsymbol{X}\right)$. 
\end{theorem}
\section{Stochastic Orderings} \label{st5}
\cite{jamali2020comparison} gave a necessary and sufficient condition for 
comparing univariate skew-normal distributions, which is as follows. 
\begin{lemma} \label{comp-dim-1}
(\cite{jamali2020comparison}) Let $X_1 \sim SN_1\left( \mu_1, \sigma_1^2, *,\delta_1\right)$ and 
$X_2 \sim SN_1\left( \mu_2, \sigma_2^2, *,\delta_2\right) $. Then, $X_1 \le_{st} X_2$ if and only if 
$\mu_1 \le \mu_2$, $\sigma_1 = \sigma_2$ and $\delta_1 \le \delta_2$. 
\end{lemma}
In the following discussion, we say $\boldsymbol{a} > \boldsymbol{b}$, where $\boldsymbol{a}$ and $\boldsymbol{b}$ are 
vectors, if and only if $a_i > b_i$ for all $i$. We say $\mathbf{A} > \mathbf{B}$, where $\mathbf{A}$ and $\mathbf{B}$ are 
matrices, if and only if $a_{ij} > b_{ij}$ for all $i$, $j$. We suppose random matrices $\boldsymbol{X} $ and $ \boldsymbol{Y}$ 
satisfy (\ref{supposexy}) and random matrices $\boldsymbol{X}_0 $ and $ \boldsymbol{Y}_0$ are such that
$\boldsymbol{X}_0 \sim SN_{n \times p} \left( \boldsymbol{0}_{n \times p}, \overline{\mathbf{\Omega}}, *,\boldsymbol{\delta}\right)$ and  
$\boldsymbol{Y}_0 \sim SN_{n \times p} \left( \boldsymbol{0}_{n \times p}, \overline{\mathbf{\Omega}'}, *,\boldsymbol{\delta}'\right)$. 
\par The following theorem, which is an extension of Lemma \ref{comp-dim-1}, provides a necessary and sufficient condition for 
comparing skew-normal matrices under stochastic order. 
\begin{theorem} \label{theorem-storder}
$\boldsymbol{X} \le_{st} \boldsymbol{Y}$ if and only if $\mathbf{M} \le \mathbf{M}'$, 
$\boldsymbol{\delta} \le \boldsymbol{\delta}'$ and $\mathbf{\Omega}=\mathbf{\Omega}'$. 
\end{theorem}
\begin{proof}
Suppose $\mathbf{M} \le \mathbf{M}'$, 
$\boldsymbol{\delta} \le \boldsymbol{\delta}'$ and $\mathbf{\Omega}=\mathbf{\Omega}'$. We know that for any increasing differentiable function $f$, 
$\frac{\partial}{\partial x_{ij}} f \left( \mathbf{X}\right)\ge 0$, $ i=1, 2, \dots, n$, $j=1, 2, \dots, p$. 
Then, the conditions of Theorem \ref{coll} are satisfied. 
\par To prove the converse, we assume $\boldsymbol{X} \le_{st} \boldsymbol{Y}$. 
It is then easy to see that $\boldsymbol{X}_{ij} \le_{st} \boldsymbol{Y}_{ij}$, $i=1, 2, \dots, n, j=1, 2, \dots, p$, 
where 
\begin{equation} \label{XY-sim}
\boldsymbol{X}_{ij} \sim SN_1\left( m_{ij}, v_{ii}\sigma_{jj}, *,\delta_{i+n\left( j-1\right)}\right) , 
\boldsymbol{Y}_{ij} \sim SN_1\left( m'_{ij}, v'_{ii}\sigma'_{jj}, *,\delta'_{i+n\left( j-1\right)}\right). 
\end{equation}
Equation (\ref{XY-sim}) is derived by Lemmas \ref{closed} and \ref{lemma-mat-skew-normal}. 
According to Lemma \ref{comp-dim-1}, we know that $m_{ij} \ge m'_{ij}$, 
$\delta_{i+n\left( j-1\right)} \ge \delta'_{i+n\left( j-1\right)}$, 
and $v_{ii}\sigma_{jj} = v'_{ii}\sigma'_{jj}$. 
\par Also we know that $\boldsymbol{X}_{ij} + \boldsymbol{X}_{kl} \le_{st} \boldsymbol{Y}_{ij} + \boldsymbol{Y}_{kl}$ 
since $\boldsymbol{X} \le_{st} \boldsymbol{Y}$, where
\begin{equation} \label{X+X}
\boldsymbol{X}_{ij} + \boldsymbol{X}_{kl} \sim SN_1\left( m_{ij} + m_{kl}, v_{ii}\sigma_{jj} + v_{kk}\sigma_{ll}+2v_{ki}\sigma_{lj}, *,\delta_{i+n\left( j-1\right)}+\delta_{k+n\left( l-1\right)}\right), 
\end{equation}
\begin{equation} \label{Y+Y}
\boldsymbol{Y}_{ij} + \boldsymbol{Y}_{kl} \sim SN_1\left( m'_{ij} + m'_{kl}, v'_{ii}\sigma'_{jj} + v'_{kk}\sigma'_{ll}+2v'_{ki}\sigma'_{lj}, *,\delta'_{i+n\left( j-1\right)}+\delta'_{k+n\left( l-1\right)}\right). 
\end{equation}
(\ref{X+X}) and (\ref{Y+Y}) are derived by using Lemmas \ref{closed} and \ref{lemma-mat-skew-normal}. 
According to Lemma \ref{comp-dim-1}, we know that $v_{ki}\sigma_{lj} = v'_{ki}\sigma'_{lj} $, which means that $\Omega=\Omega'$. 
\end{proof}
\begin{remark}
Suppose $\mathbf{\Omega} = \mathbf{V} \otimes \boldsymbol{\Sigma}$ and $\mathbf{\Omega}' = \mathbf{V}' \otimes \boldsymbol{\Sigma}'$. Then, we have $\mathbf{\Omega} = \mathbf{\Omega}'$ 
if and only if there exists $a \in \mathbb{R}$, $a \neq 0$, such that
\begin{equation}
\mathbf{V}=a\mathbf{V}', \boldsymbol{\Sigma}=\frac{1}{a} \boldsymbol{\Sigma}'. 
\end{equation}
This is called the uniqueness of Kronecker product. As a result, it is clear that 
$\boldsymbol{X} \le_{st} \boldsymbol{Y}$ if and only if $\mathbf{M} \le \mathbf{M}'$, 
$\boldsymbol{\delta} \le \boldsymbol{\delta}'$, and that there exists $a \neq 0$ such that $\mathbf{V}=a\mathbf{V}'$ and $\boldsymbol{\Sigma}=\frac{1}{a} \boldsymbol{\Sigma}'$. 
\end{remark}
Conditions for convex order of matrix variate skew-normal distributions are deduced from the following theorem. 
We consider standardized skew-normal matrices $\boldsymbol{X}_0$ and $\boldsymbol{Y}_0$ in Case 2
since the location parameter $\mathbf{M}$ would abate the simplicity of results. 
\begin{theorem} \label{trick}
\begin{enumerate}
\item If $\mathbf{M} = \mathbf{M}'$, 
$\boldsymbol{\delta} = \boldsymbol{\delta}'$ and $\mathbf{\Omega}' - \mathbf{\Omega}$ is positive semi-definite, then 
$\boldsymbol{X} \le_{cx} \boldsymbol{Y}$;
\item $\boldsymbol{X}_0 \le_{cx} \boldsymbol{Y}_0$ if and only if 
$\boldsymbol{\delta} = \boldsymbol{\delta}'$ and $\overline{\mathbf{\Omega}'} - \overline{\mathbf{\Omega}}$ is positive semi-definite. 
\end{enumerate}
\end{theorem}
\begin{proof}
1. If convex function $f$ is twice differentiable, then Hessian matrix $\boldsymbol{H}_f\left( \mathbf{X}\right)$ is 
positive semi-definite. As $\mathbf{\Omega}' - \mathbf{\Omega}$ is positive semi-definite, there
exists a matrix $\mathbf{A}_{np \times np}$ such that $\mathbf{\Omega}' - \mathbf{\Omega} = \mathbf{AA}^T$. Suppose 
$\mathbf{A}=\left( \boldsymbol{a}_1, \boldsymbol{a}_2, \dots, \boldsymbol{a}_{np}\right)$, where $\boldsymbol{a}_i$ is an $np$-dimensional column vector, for $i=1, 2, \dots, np$. Then, according to the properties of trace of matrices, we have
\begin{equation}
{\rm tr}\left( \left( \mathbf{\Omega} - \mathbf{\Omega}' \right)\boldsymbol{H}_f\left( \mathbf{X}\right)\right) = {\rm tr}\left( \mathbf{A}^T \boldsymbol{H}_f\left( \mathbf{X}\right) \mathbf{A}\right) =\sum_{i=1}^{np} \boldsymbol{a}_i^T \boldsymbol{H}_f\left( \mathbf{X}\right) \boldsymbol{a}_i \ge 0. 
\end{equation}
According to Theorem \ref{coll}, we have $Ef\left( \boldsymbol{Y}\right) - Ef\left( \boldsymbol{X}\right) \ge 0$
for all twice differentiable convex functions $f$, which means $\boldsymbol{X} \le_{cx} \boldsymbol{Y}$. 
\par 2. The proof of suffciency is obvious, and so we just prove the necessity. By letting $f\left( \mathbf{X}\right)=x_{ij}$
and $f\left( \mathbf{X}\right)=-x_{ij}$ and using Definitions \ref{def-F-order} and \ref{def-order}, 
we can easily see that the mean matrices of $\boldsymbol{X}_0$ and $\boldsymbol{Y}_0$ are identical. 
Then, according to Theorem \ref{mean-cov}, $E\left( \boldsymbol{X}_0\right) = E\left( \boldsymbol{Y}_0\right)$
which means $\boldsymbol{\delta} = \boldsymbol{\delta}'$. 
\par If $\overline{\mathbf{\Omega}}' - \overline{\mathbf{\Omega}}$ is not positive semi-definite, then there exists an 
$np$-dimensional column vector $\boldsymbol{b}$ such that $\boldsymbol{b}^T \overline{\mathbf{\Omega}} \boldsymbol{b} > \boldsymbol{b}^T \overline{\mathbf{\Omega}'} \boldsymbol{b}$. 
Let $f\left( \mathbf{X}\right)=\left( \boldsymbol{b}^T {\rm vec}\left(\mathbf{X}\right)\right)^2=\boldsymbol{b}^T {\rm vec}\left( \mathbf{X} \right) {\rm vec}\left( \mathbf{X} \right)^T \boldsymbol{b}$. 
Then, $f\left( \mathbf{X}\right)$ is obviously a convex function. Using Theorem \ref{mean-cov}, we have 
\begin{equation}
E\left( f(\boldsymbol{Y}_0)\right)-E\left( f(\boldsymbol{X}_0)\right) = \boldsymbol{b}^T \overline{\mathbf{\Omega}'}\boldsymbol{b} - \boldsymbol{b}^T \overline{\mathbf{\Omega}}\boldsymbol{b} <0, 
\end{equation}
which is a contradiction to $\boldsymbol{X}_0 \le_{cx} \boldsymbol{Y}_0$. 
\end{proof}
The following theorem provides conditions for comparing skew-normal random matrices under increasing convex order. 
We recall that a matrix $\mathbf{A}$ is said to be copositive if $ \boldsymbol{x}^T \mathbf{A} \boldsymbol{x} \ge 0$ 
for all $\boldsymbol{x} \ge 0$; 
see \cite{hadeler1983copositive} for details. 
\begin{theorem}
\begin{enumerate}
\item If $\mathbf{M} \le \mathbf{M}'$, 
$\boldsymbol{\delta} \le \boldsymbol{\delta}'$ and $\mathbf{\Omega}' - \mathbf{\Omega}$ is positive semi-definite, then 
$\boldsymbol{X} \le_{icx} \boldsymbol{Y}$;
\item $\boldsymbol{X}_0 \le_{icx} \boldsymbol{Y}_0$ if and only if 
$\boldsymbol{\delta} \le \boldsymbol{\delta}'$ and $\overline{\mathbf{\Omega}'} - \overline{\mathbf{\Omega}}$ is copositive. 
\end{enumerate}
\end{theorem}
\begin{proof}
1. If $f$ is an increasing convex function, then we have $\frac{\partial}{\partial x_{ij}} f\left( \mathbf{X}\right) \ge 0$
and $H_f\left(\mathbf{X} \right)$ is positive semi-definite. As we know that $\mathbf{M} \le \mathbf{M}'$, 
$\boldsymbol{\delta} \le \boldsymbol{\delta}'$ and $\overline{\mathbf{\Omega}} - \overline{\mathbf{\Omega}}'$ is positive semi-definite. Then, the conditions of Theorem \ref{coll}
are satisfied, which means that $\boldsymbol{X} \le_{icx} \boldsymbol{Y}$. 
\par 2. By letting $f\left( \mathbf{X}\right)=x_{ij}$, which is an increasing convex function, we see
that $E\boldsymbol{X}_0 \le E\boldsymbol{Y}_0$. Using Theorem \ref{mean-cov}, we get $\boldsymbol{\delta} \le \boldsymbol{\delta}'$. 
\par Let $f\left( \mathbf{X}\right) = \boldsymbol{a}^T {\rm vec}\left( \mathbf{X} \right)$, $\boldsymbol{a} \ge \boldsymbol{0} $ and $ \boldsymbol{a} \neq \boldsymbol{0}$ , which 
is an increasing convex function. Then, we know that $\overline{X}_0 \le_{icx} \overline{Y}_0$, where
\begin{equation} \label{X_0}
\overline{\boldsymbol{X}}_0 = \boldsymbol{a}^T {\rm vec}\left( \boldsymbol{X}_0 \right) \sim SN_1\left( 0, \boldsymbol{a}^T \overline{\mathbf{\Omega}} \boldsymbol{a}, *,\boldsymbol{a}^T \boldsymbol{\delta}\right), 
\end{equation}
\begin{equation} \label{Y_0}
\overline{\boldsymbol{Y}}_0 = \boldsymbol{a}^T {\rm vec}\left( \boldsymbol{Y}_0 \right) \sim SN_1\left( 0, \boldsymbol{a}^T \overline{\mathbf{\Omega}'} \boldsymbol{a}, *,\boldsymbol{a}^T \boldsymbol{\delta}'\right), 
\end{equation}
which are derived by using Lemma \ref{closed}. Let $\sigma_0^2=\boldsymbol{a}^T \overline{\mathbf{\Omega}} \boldsymbol{a}$ and $\left( \sigma'_0\right)^2=\boldsymbol{a}^T \overline{\mathbf{\Omega}'} \boldsymbol{a}$. We now claim that 
$\sigma_0 \le \sigma'_0$. If $\sigma_0 > \sigma'_0$, then
\begin{equation}
\begin{split}
\lim_{t \to + \infty} \frac{E\left( \boldsymbol{Y}_0 - t\right)_+}{E\left( \boldsymbol{X}_0 - t\right)_+}&=\lim_{t \to + \infty} \frac{\int_t^{+\infty} 1-F_{Y_0} \left( x\right) dx}{\int_t^{+\infty} 1-F_{X_0} \left( x\right) dx} \\
&=\lim_{t \to + \infty} \frac{F_{Y_0} \left( t\right)-1}{F_{X_0} \left( t\right)-1}\\
&=\lim_{t \to + \infty} \frac{f_{Y_0} \left( t\right)}{f_{X_0} \left( t\right)} = 0, 
\end{split}
\end{equation}
where the last equality comes from Definition \ref{n-dim-sn-def}. 
This contradicts the fact that $\overline{\boldsymbol{X}}_0 \le_{icx} \overline{\boldsymbol{Y}}_0$ since both $E\left( \boldsymbol{Y}_0 - t\right)_+$ and $E\left( \boldsymbol{X}_0 - t\right)_+$ are nonnegative. 
This means that $\overline{\mathbf{\Omega}} - \overline{\boldsymbol{\Omega'}}$ is copositive. 
\end{proof}
For directionally convex order, we have the following result.
\begin{theorem}
$\boldsymbol{X}_0 \le_{dcx} \boldsymbol{Y}_0$ if and only if 
$\boldsymbol{\delta} = \boldsymbol{\delta}'$ and $\overline{\mathbf{\Omega}'} - \overline{\mathbf{\Omega}} \geq \boldsymbol{0}$. 
\end{theorem}
\begin{proof}
As the suffciency is obvious, we just prove the necessity. $\boldsymbol{\delta} = \boldsymbol{\delta}'$ can be proved by using the same argument as in 
the proof of Theorem \ref{trick}, since both $f\left( \mathbf{X}\right)=x_{ij}$
and $f\left( \mathbf{X}\right)=-x_{ij}$ are directionally convex functions. 
\par $E\left( \boldsymbol{X}_{0ij}\boldsymbol{X}_{0kl}\right) \le E\left( \boldsymbol{Y}_{0ij}\boldsymbol{Y}_{0kl}\right)$ can be derived by letting $f\left( \mathbf{X}\right)=x_{ij}x_{kl}$. 
Thus, it can be shown that $E\left({\rm vec}\left(\boldsymbol{X}_0\right){\rm vec}\left(\boldsymbol{X}_0\right)^T\right) $ $\le$ $ E\left({\rm vec}\left(\boldsymbol{Y}_0\right){\rm vec}\left(\boldsymbol{Y}_0\right)^T\right)$, which yields $\overline{\mathbf{\Omega}'} - \overline{\mathbf{\Omega}} \geq \boldsymbol{0}$
by using Lemma \ref{mean-cov}. 
\end{proof}
\cite{azzalini2013skew} have presented a stochastic representation of multivariate skew-normal distribution as follows.
\begin{lemma} \label{lemma-2013}
(\cite{azzalini2013skew}) If $\boldsymbol{X} \sim SN_n\left(\boldsymbol{0}, \overline{\mathbf{\Omega}},\boldsymbol{\alpha},*\right)$, 
then $\boldsymbol{X} \overset{d}{=} \boldsymbol{U} | \lbrace V < \boldsymbol{\alpha}^T \boldsymbol{U}\rbrace$, where $\boldsymbol{U} \sim N_n\left( \boldsymbol{0}, \overline{\mathbf{\Omega}}\right)$ and is 
independent of the random variable $V \sim N\left(0, 1\right)$.
\end{lemma} 
\begin{theorem} \label{theorem-2013}
If $\boldsymbol{X} \sim SN_n\left( \boldsymbol{\mu}, \mathbf{\Omega}, \boldsymbol{\alpha},*\right)$, 
then $\boldsymbol{X} \overset{d}{=} \boldsymbol{\mu} + \boldsymbol{U} | \lbrace V < \boldsymbol{\alpha}^T \boldsymbol{\omega}^{-1} \boldsymbol{U}\rbrace$, where $\boldsymbol{U} \sim N_n\left( \boldsymbol{0}, \mathbf{\Omega}\right)$ and is 
independent of the random variable $V \sim N\left(0, 1\right)$.     
\end{theorem} 
\begin{proof}
We know that if $\boldsymbol{U} \sim N_n\left( \boldsymbol{0}, \mathbf{\Omega}\right)$, then $\boldsymbol{\omega}^{-1}\boldsymbol{U} \sim N_n\left( \boldsymbol{0}, \overline{\mathbf{\Omega}}\right)$. 
According to Lemma \ref{lemma-2013}, we have $\boldsymbol{\omega}^{-1}\boldsymbol{U}| \lbrace V < \boldsymbol{\alpha}^T \boldsymbol{\omega}^{-1} \boldsymbol{U}\rbrace \sim SN_n\left(\boldsymbol{0}, \overline{\mathbf{\Omega}},\boldsymbol{\alpha},*\right)$, 
and according to Remark \ref{A-eq-w}, $\boldsymbol{\omega} \boldsymbol{\omega}^{-1}\boldsymbol{U}| \lbrace V < \boldsymbol{\alpha}^T \boldsymbol{\omega}^{-1} \boldsymbol{U}\rbrace \sim SN_n\left(\boldsymbol{0}, \mathbf{\Omega},\boldsymbol{\alpha},*\right)$.
\end{proof}
The upper orthant order, given in Definition \ref{def-order}, can also be defined through a comparison of upper orthants. 
Specifically, $\boldsymbol{X} \leq_{uo} \boldsymbol{Y}$ holds if and only if 
$P\left( \boldsymbol{X} > \mathbf{T}\right) \le P\left( \boldsymbol{Y} > \mathbf{T}\right)$ holds for all matrices $\mathbf{T}$. 
The two definitions can be shown to be equivalent. 
The following theorem provides conditions for comparing matrix variate skew-normal distributions under upper orthant order.
\begin{theorem}
\begin{enumerate}
\item If $\mathbf{M} = \mathbf{M}'$, $\boldsymbol{\delta} \le \boldsymbol{\delta}'$, $\omega_{ii}=\omega'_{ii}$ and $ \omega_{ij} \le \omega'_{ij}$, then 
$\boldsymbol{X} \le_{uo} \boldsymbol{Y}$;
\item If $\boldsymbol{X} \le_{uo} \boldsymbol{Y}$, then $\mathbf{M} \le \mathbf{M}'$, $\boldsymbol{\delta} \le \boldsymbol{\delta}'$ and $ \omega_{ii}=\omega'_{ii}$. 
\end{enumerate}
\end{theorem}
\begin{proof}
1. Let $\boldsymbol{X}$ and $\boldsymbol{Y}$ be distributed as in (\ref{supposexy}), and that $\boldsymbol{X}$ and $\boldsymbol{Y}$ have skewness parameters $\mathbf{B}$ and $\mathbf{B}'$, respectively. Then, we have 
${\rm vec}\left( \boldsymbol{X}\right) \sim SN_{np}\left({\rm vec}\left( \mathbf{M}\right), \mathbf{\Omega},{\rm vec}\left( \mathbf{B}\right), \boldsymbol{\delta} \right)$ 
and ${\rm vec}\left( \boldsymbol{Y}\right) \sim SN_{np} ( {\rm vec}\left( \mathbf{M}'\right), \mathbf{\Omega}',$ $ {\rm vec}\left( \mathbf{B}'\right),\boldsymbol{\delta}' )$ by using Lemma \ref{lemma-mat-skew-normal} and Remark \ref{share-delta}. 
According to Lemma \ref{lemma-2013}, we have 
${\rm vec}\left( \boldsymbol{X}\right) \overset{d}{=} {\rm vec}\left( \mathbf{M}\right) + \boldsymbol{U} | \lbrace V< {\rm vec}\left(\mathbf{B}\right)^T \boldsymbol{\omega}^{-1} \boldsymbol{U}\rbrace$ 
and ${\rm vec}\left( \boldsymbol{Y}\right) \overset{d}{=} {\rm vec}\left( \mathbf{M}'\right) + \boldsymbol{U}' | \lbrace V'<{\rm vec}\left( \mathbf{B}'\right)^T \boldsymbol{\omega}'^{-1} \boldsymbol{U}'\rbrace$. 
Then, 
\begin{equation}
P\left({\rm vec}\left( \boldsymbol{X}\right) > \boldsymbol{t}\right)=2P\left(\boldsymbol{Z}>\left( \boldsymbol{t} - {\rm vec}\left( \mathbf{M}\right), 0\right)\right), 
P\left({\rm vec}\left( \boldsymbol{Y}\right) > \boldsymbol{t}\right)=2P\left(\boldsymbol{Z}'>\left( \boldsymbol{t} - {\rm vec}\left( \mathbf{M}'\right), 0\right)\right), 
\end{equation}
where 
\begin{equation}
\boldsymbol{Z} \sim N_{np+1} \left( \boldsymbol{0}, \begin{pmatrix}
\mathbf{\Omega} & \boldsymbol{\delta} \\ \boldsymbol{\delta}^T & 1
\end{pmatrix} 
\right), 
\boldsymbol{Z}' \sim N_{np+1} \left( \boldsymbol{0}, \begin{pmatrix}
\mathbf{\Omega}' & \boldsymbol{\delta}' \\ \boldsymbol{\delta}'^T & 1
\end{pmatrix} 
\right).
\end{equation}
Considering the conditions $\mathbf{M} = \mathbf{M}'$, $\boldsymbol{\delta} \le \boldsymbol{\delta}'$, 
$ \omega_{ii}=\omega'_{ii}$ and $ \omega_{ij} \le \omega'_{ij}$ and using Slepian's inequality (Theorem 2.1.1 in \cite{tong2014probability}), 
we can conclude that $P\left( \boldsymbol{X} > \mathbf{T}\right) \le P\left( \boldsymbol{Y} > \mathbf{T}\right)$, 
which means $\boldsymbol{X} \le_{uo} \boldsymbol{Y}$. 
\par 2. $\boldsymbol{X}_{ij} \le_{st} \boldsymbol{Y}_{ij}$ can be 
obtained from $\boldsymbol{X} \le_{uo} \boldsymbol{Y}$ due to the equivalence between stochastic order and upper orthant order in the univariate case. 
Then, $\mathbf{M} \le \mathbf{M}'$, $\boldsymbol{\delta} \le \boldsymbol{\delta}'$ and $ \omega_{ii}=\omega'_{ii}$ 
can be proved by using the same idea as in the proof of Theorem \ref{theorem-storder}. 
\end{proof}
\begin{remark}
Proposition 4.6 of \cite{jamali2020comparison} provides some conditions for comparing multivariate skew-normal distributions
under upper orthant order. 
We must point out that Part (i) in Proposition 4.6 of \cite{jamali2020comparison} seems to be in error due to the misuse of Slepian's inequality.
The correct proposition should in fact be as follows: Suppose $\boldsymbol{X} \sim SN_n\left(\xi,\boldsymbol{\Omega},*,\boldsymbol{\delta}\right)$ and  
$\boldsymbol{Y} \sim SN_n\left(\xi',\boldsymbol{\Omega}',*,\boldsymbol{\delta}'\right)$. 
If $\xi = \xi'$, $\boldsymbol{\delta} \le \boldsymbol{\delta}'$, 
$\omega_{ii}=\omega'_{ii}$ and $ \omega_{ij} \le \omega'_{ij}$, then 
$\boldsymbol{X} \le_{uo} \boldsymbol{Y}$.
\end{remark}
The following result generalizes Theorem 11 in \cite{muller2001stochastic} for the multivariate normal case
to the setting considered here. 
\begin{theorem}
$\boldsymbol{X}_0 \le_{sm} \boldsymbol{Y}_0$ if and only if $\boldsymbol{X}_0$ and $\boldsymbol{Y}_0$ have the same 
marginals and $\omega_{ij} \le \omega'_{ij}$. 
\end{theorem}
\begin{proof}
If $\boldsymbol{X}_0 \le_{sm} \boldsymbol{Y}_0$, then $\boldsymbol{X}_0$ and $\boldsymbol{Y}_0$ have 
the same marginals (see \cite{muller2000some}). This
yields $\boldsymbol{\delta} = \boldsymbol{\delta}'$ since 
\begin{equation} \label{XYsim}
\boldsymbol{X}_{0ij} \sim SN_1\left(0,v_{ii}\sigma_{jj}, *,\delta_{i+n\left( j-1\right)}\right) , 
\boldsymbol{Y}_{0ij} \sim SN_1\left(0,v'_{ii}\sigma'_{jj}, *,\delta'_{i+n\left( j-1\right)}\right). 
\end{equation}
We can prove $E\left(\boldsymbol{X}_{0ij}\boldsymbol{X}_{0kl}\right) \le E\left(\boldsymbol{Y}_{0ij}\boldsymbol{Y}_{0kl}\right)$ by letting $f\left( \mathbf{X}\right)=x_{ij}x_{kl}$, which is a supermodular function. 
Then, $E\left({\rm vec}\left(\boldsymbol{X}_0\right){\rm vec}\left(\boldsymbol{X}_0\right)^T\right) \le E\left({\rm vec}\left( \boldsymbol{Y}_0\right){\rm vec}\left( \boldsymbol{Y}_0\right)^T\right)$, which yields that $\omega_{ij} \le \omega'_{ij}$
by using Theorem \ref{mean-cov}. 
\par Conversely, if $\boldsymbol{X}_0$ and $\boldsymbol{Y}_0$ have the same 
marginals and $\omega_{ij} \le \omega'_{ij}$, then $\boldsymbol{X}_0 \le_{sm} \boldsymbol{Y}_0$ can be obtained from 
Theorem \ref{coll}, since $\frac{\partial^2}{\partial x_{ij} \partial x_{kl}} f\left( \mathbf{X}\right) \ge 0$ for
supermodular function $f$ and $\boldsymbol{\delta} = \boldsymbol{\delta}'$. 
\end{proof}
\section{Concluding Remarks} \label{CR6}
In this paper, the characteristic function of the matrix variate skew-normal distribution has been derived, and 
has been utilized to establish integral stochastic orderings of matrix variate skew-normal distributions. 
Considering six different important integral stochastic orderings, some necessary and sufficient conditions have been derived. 
It will naturally be of interest to further generalize these results to matrix variate skew-elliptical family of distributions. 
We are currently working on this problem and hope to report the findings in a future paper.

\section*{Acknowledgements}
This research was supported by the National Natural Science Foundation of China (No. 12071251, 11571198, 11701319). 


\end{document}